\documentclass[10pt]{amsart}
\usepackage{amssymb}
\newtheorem{define}{Definition}

\newtheorem{lemma}{Lemma}

\newtheorem{theorem}{Theorem}
\newtheorem{remark}{Remark}

\begin{document}
\title [Functional inequalities for the Wright function]{ Monotonicity patterns and functional inequalities FOR CLASSICAL AND GENERALIZED
Wright FUNCTIONS \\}%
\vspace{.5cm}
\author[ K. Mehrez]{KHALED MEHREZ}
\address{Khaled Mehrez. D\'epartement de Math\'ematiques ISSAT Kasserine, Universit\'e de Kairouan, Tunisia.}
 \email{k.mehrez@yahoo.fr}
\begin{abstract}
 In this paper our aim is to present the completely monotonicity and convexity properties for the Wright function. As consequences of these results, we present some functional inequalities. Moreover, we derive the  monotonicity and log-convexity results for the generalized Wright functions. As applications, we present several new inequalities (like Tur\'an type inequalities) and we prove some geometric properties for four--parametric Mittag--Leffler functions.  
\end{abstract}
\maketitle
%%%%%%%%%%%%%%%%%%%%%%%%%%%%%%%%%%%%%%%%%%%%%%%%%%%%%%%
\noindent{\textbf{ Keywords:}} Wright function, generalized Wright function, four--parametric Mittag--Leffler function, complete monotonicity, Log--convexity, Tur\'an type inequalities.

\noindent \textbf{Mathematics Subject Classification (2010)}: 33C20; 33E12; 26D07.
\section{Introduction}

Special functions like  Mittag--Leffler functions and Wright functions $E_{\alpha,\beta}(z)$ and $W_{\alpha,\beta}(z)$ are frequently  in the solution of linear partial fractional differential equations,  the number theory regarding the asymptotic of the number of some special partitions of the natural numbers and in the boundary--value problems for the fractional diffusion-wave equation,
that is, the linear partial integro-differential equation obtained
from the classical diffusion or wave equation by replacing
the first- or second order time derivative by a fractional
derivative of order 𝛼 with $0<\alpha<2$, it was found that the
corresponding Green functions can be represented in terms of the  Wright function. This special function
are related to modified Bessel functions of the first kind, and thus their properties can be
useful in problems of mathematical physics. 

The Wright function is defined by the series representation, valid in the whole complex plane
\begin{equation}
W_{\alpha,\beta}(z)=\sum_{k=0}^\infty\frac{z^k}{k!\Gamma(\alpha k+\beta)},\;\alpha>-1,\beta\in\mathbb{C}.
\end{equation}
It is an entire function of order $1/(1+\alpha),$ which has been known also as generalized Bessel (or Bessel Maitland) function. \\

Our aim in this paper is twofold: in one hand is to prove the completely monotonicity properties for the Wright function $W_{\alpha,\beta}(-z)$ for $\alpha,\beta>,0$ and $0<z<1$. As consequence, we derive  some functional inequalities as well as lower and upper bounds for the Wright function. On the other hand, by using the completely monotonicity property for the classical Wright function we obtain the completely monotonicity  for the generalized Wright function, and consequently we get also the monotonicity property for the four--parametric Mittag--Leffler function.\\

The present sequel to some of the aforementioned investigations is organized as follows. In section 2, we present new integral representation for the Wright function. Moreover, we derive some monotonicity and convexity results for the function $z\mapsto W_{\alpha,\beta}(-z)$. As a consequence, we establish a number of functional inequalities. In section 3, the monotonicity property for generalized Wright function is proved. As applications, we prove several new inequalities for this functions. In particular, we gave some Tur\'an type inequalities for the generalized Wright function. Finally, in section 4, we  apply some of our main results of Section 3 with a view to deriving some new inequalities for the four--parametric Mittag--Leffler function.\\

Each of the following definitions will be used in our investigation.\\

\begin{define}A function $f:(0,\infty)\subseteq\mathbb{R}\rightarrow\mathbb{R}$ is said to be completely monotonic if $f$ has derivatives of all orders and satisfies the following inequalities:
$$(-1)^nf^{(n)}(x)\geq0,\;\;\left(x>0,\;\;\textrm{and}\;\;n\in\mathbb{N}=\left\{1,2,3,...\right\}\right).$$
\end{define}

\begin{define} A function $f:[a,b]\subseteq\mathbb{R}\rightarrow\mathbb{R}$  is said to be log-convex if its natural logarithm $\log f$ is convex, that is, for all $x,y\in[a,b]$ and $\alpha\in[0,1]$ we have
$$f(\alpha x+(1-\alpha)y)\leq[f(x)]^\alpha[f(y)]^{1-\alpha}.$$
If the above inequality is reversed then $f$ is called a log-concave function. It is also known that if $g$ is differentiable, then $f$ is log-convex (log-concave) if and only if $f^\prime/f$ is increasing (decreasing). 
\end{define}
\section{The Wright functions: Monotonicity patterns and functional inequalities}

In the next Lemma we present new integral representation for the Wright function $W_{\alpha,\beta}(z).$

\begin{lemma}\label{l1}Let $\beta>\alpha>0.$ Then the the Wright function $W_{\alpha,\beta}(z)$ has the following integral representation
\begin{equation}\label{1}
W_{\alpha,\beta}(z)=c_{\alpha,\beta}\int_0^1(1-t^{1/\alpha})^{\beta-\alpha-1}W_{\alpha,\alpha}(zt)dt,\;z\in\mathbb{R},
\end{equation}
where $c_{\alpha,\beta}=\frac{1}{\alpha\Gamma(\beta-\alpha)}.$ In particular,
$$W_{\alpha,\alpha+1}(z)=\frac{1}{\alpha}\int_0^1W_{\alpha,\alpha}(zt)dt.$$
\end{lemma}
\begin{proof}By using the definition of the Wright function $W_{\alpha,\beta}(z)$, we get
\begin{equation*}
\begin{split}
\int_0^1(1-t^{1/\alpha})^{\beta-\alpha-1}W_{\alpha,\alpha}(z)dt&=\int_0^1(1-t^{1/\alpha})^{\beta-\alpha-1}\sum_{k=0}^\infty \frac{(zt)^k}{k!\Gamma(\alpha+k\alpha)}dt\\
&=\sum_{k=0}^\infty \frac{1}{k!\Gamma(\alpha+k\alpha)}\left(\int_0^1(1-t^{1/\alpha})^{\beta-\alpha-1}t^kdt\right)z^k\\
&=\alpha\sum_{k=0}^\infty \frac{1}{k!\Gamma(\alpha+k\alpha)}\left(\int_0^1(1-t)^{\beta-\alpha-1}t^{\alpha k+\alpha-1}dt\right)z^k\\
&=\alpha\sum_{k=0}^\infty \frac{B(\beta-\alpha,\alpha k+\alpha)}{k!\Gamma(\alpha+k\alpha)}z^k\\
&=\frac{W_{\alpha,\beta}(z)}{c_{\alpha,\beta}},
\end{split}
\end{equation*}
where $B(x,y)$ is the Beta function defined by $B(x,y)=\int_0^1t^{x-1}(1-t)^{y-1}dt=\frac{\Gamma(x)\Gamma(y)}{\Gamma(x+y)}.$ Finally, letting in (\ref{1}) the value $\beta=\alpha+1$ we obtain the integral representation for  the Wright function $W_{\alpha,\alpha+1}(z).$ 
\end{proof}
\begin{lemma}\label{L2}Let $\alpha>0$ and $\beta>x^\star,$ where $x^\star\simeq 1.461632144...$ is the abscissa of the minimum of the Gamma function, Then the function $W_{\alpha,\beta}(-z)$ is nonnegative for all $z\in(0,1).$
\end{lemma}
\begin{proof}Let $u_k(z)=\frac{z^k}{k!\Gamma(\alpha k+\beta)}$, we get 
\begin{equation}\label{K}
W_{\alpha,\beta}(-z)=u_0(z)-u_1(z)+\sum_{k=2}^\infty (-1)^k u_k(z).
\end{equation} 
Elementary calculations reveal that for $0<z<1,$ and $k\geq2$
\begin{equation}\label{K0}
\frac{u_{k+1}(z)}{u_{k}(z)}=\frac{\Gamma(\alpha k+\beta)z}{(k+1)\Gamma(\alpha k+\beta+\alpha)}\leq\frac{\Gamma(\alpha k+\beta)}{\Gamma(\alpha k+\beta+\alpha)}.
\end{equation}
From the previous inequality and using the fact $z\mapsto\Gamma(z)$ is increasing on $(x^\star,\infty)$  we deduce that $\frac{u_{k+1}(z)}{u_k(z)}\leq1.$ Therefore, for fixed $0<z<1,$ the sequence $k\mapsto u_k(z)$ is decreasing with regards $k\geq 2$ and $u_k$ tends to $0$ as $k\longrightarrow\infty.$  From (\ref{K}) and since the Gamma function is increasing on $(x^\star,\infty)$ and  we have
\begin{equation*}
\begin{split}
W_{\alpha,\beta}(-z)&\geq u_0(z)-u_1(z)=\frac{1}{\Gamma(\beta)}-\frac{z}{\Gamma(\beta+\alpha)}\\
&\geq\frac{1}{\Gamma(\beta)}-\frac{1}{\Gamma(\beta+\alpha)}\geq0.
\end{split}
\end{equation*}
 The proof of Lemma \ref{L2} is complete. 
\end{proof}

\begin{theorem}\label{T1} Let $\beta>\alpha>x^\star$. Then, the function $z\mapsto \check{W}_{\alpha,\beta}(z)=W_{\alpha,\beta}(-z)$ is completely monotonic and log-convex on $(0,1).$ Furthermore, the following inequalities
\begin{equation}\label{21}
\check{W}_{\alpha,\beta}(x+y)\geq\frac{\check{W}_{\alpha,\beta}(x)\check{W}_{\alpha,\beta}(y)}{\Gamma(\beta)},\;\;0<x+y<1,
\end{equation}
\begin{equation}\label{22}
\check{W}_{\alpha,\beta+2\alpha}(z)\check{W}_{\alpha,\beta}(z)-\Big(\check{W}_{\alpha,\beta+\alpha}(z)\Big)^2\geq0,\;0<z<1,
\end{equation}
\begin{equation}\label{23}
\check{W}_{\alpha,\beta}(z)\geq\frac{e^{-{\frac{\Gamma(\beta)}{\Gamma(\beta+\alpha)}}z}}{\Gamma(\beta)},\;0<z<1,
\end{equation}\label{22}
are valid .
\end{theorem}
\begin{proof}By using the differentiation formula 
\begin{equation}
\frac{d}{dz}W_{\alpha,\beta}(z)=W_{\alpha,\beta+\alpha}(z),
\end{equation}
Lemma \ref{l1} and Lemma \ref{L2}, we have for $n\in\mathbb{N}$ and $\beta>\alpha>0,$ 
$$(-1)^n\Big(\check{W}_{\alpha,\beta}(z)\Big)^{(n)}=c_{\alpha,\beta}\int_0^1(1-t^{1/\alpha})^{\beta-\alpha-1}t^n \check{W}_{\alpha,\alpha+n\alpha}(zt)dt\geq0,$$
for all $z\in(0,1)$. Thus, the function $z\mapsto \check{W}_{\alpha,\beta}(z)$ is completely monotonic and consequently is log-convex, since every completely monotonic function is log--convex, see \cite[p.167]{WI}. It is clear that the function $z\mapsto\mathcal{\check{W}}_{\alpha,\beta}(z)=\Gamma(\beta)\check{W}_{\alpha,\beta}(z)$ maps $(0,1)$ to $(0,1)$ and it is completely monotonic on $(0,\infty)$ for all $\beta>\alpha>0.$ On the other hand, according to Kimberling \cite{KI} if a function $f,$ defined on $(0,\infty)$ is continuous and completely monotonic and maps $(0,\infty)$ to $(0,1)$, then $\log f$ is super--additive, that is for all $0<x,y<1$ we have
$$f(x+y)\geq f(x)f(y).$$
Therefore, we conclude the asserted inequality (\ref{21}). Now, focus on the Tur\'an type inequality (\ref{22}). Since the function $z\mapsto\check{W}_{\alpha,\beta}(z)$ is log-convex on $(0, 1)$, it follows that the function $z\mapsto\check{W}_{\alpha,\beta}^\prime(z)/\check{W}_{\alpha,\beta}(z)$ is increasing on $(0,1).$ Thus
$$\left(\frac{\check{W}_{\alpha,\beta}^\prime(z)}{\check{W}_{\alpha,\beta}(z)}\right)^\prime=\frac{\check{W}_{\alpha,\beta+2\alpha}(z)\check{W}_{\alpha,\beta}(z)-\Big(\check{W}_{\alpha,\beta+\alpha}(z)\Big)^2}{\check{W}_{\alpha,\beta}^2(z)}\geq0.$$ Next, to prove the inequality (\ref{23}), we set 
$$F(x)=\log\Big(\Gamma(\beta)\check{W}_{\alpha,\beta}(z)\Big)\;\;\textrm{and}\;\;G(x)=x.$$
By usnig the fact that $z\mapsto\check{W}_{\alpha,\beta}^\prime(z)/\check{W}_{\alpha,\beta}(z)$ is increasing on $(0,\infty)$ and  monotone form of l'Hospital's rule \cite{AN}, we deduce that the function $z\mapsto F(z)/G(z)=(F(z)-F(0))/(G(z)-G(0))$ is increasing on $(0,1),$ and consequently 
$$\frac{F(z)}{G(z)}\geq \lim_{x\longrightarrow0}F^\prime(z)=-\frac{\Gamma(\beta)}{\Gamma(\alpha+\beta)}.$$
This completes the proof of the Theorem \ref{T1}.
\end{proof}
\begin{theorem}The following inequalities holds true:\\
\textbf{a.} For $\beta-\alpha\geq1$ and $z>0$, we have:
\begin{equation}\label{OOO}
W_{\alpha,\beta}(z)\leq\left(\frac{\Gamma(2\alpha)}{\Gamma^2(\beta)}\right).\left(\frac{e^{\frac{\Gamma(\alpha)z}{\Gamma(2\alpha)}}-1}{z}\right).
\end{equation}
\textbf{b.} For $\beta-\alpha\geq2$ and $z>0$, we have:
\begin{equation}\label{gg}
W_{\alpha,\beta+1}(z)W_{\alpha,\beta-1}(z)\leq \frac{\Gamma(\beta-\alpha)}{\Gamma(\beta-\alpha-1)\Gamma(\beta-\alpha+1)}W_{\alpha,\alpha+1}(z)W_{\alpha,\beta}(z).
\end{equation}
 In particular, we get
\begin{equation}\label{ggg}
2 W_{\alpha,\alpha+3}(z)\leq W_{\alpha,\alpha+2}(z).
\end{equation}
\end{theorem}
\begin{proof} \textbf{a.} In \cite[Theorem 6.1]{Khaled1} the author proved that
\begin{equation}\label{L}
W_{\alpha,\beta}(z)\leq \frac{e^{\frac{\Gamma(\beta)}{\Gamma(\beta+\alpha)}z}}{\Gamma(\beta)},\;z>0.
\end{equation}
In view of (\ref{1}) and (\ref{L}), we obtain  
\begin{equation}\label{55}
W_{\alpha,\beta}(z)\leq \frac{c_{\alpha,\beta}}{\Gamma(\beta)}\int_0^1(1-t^{1/\alpha})^{\beta-\alpha-1}e^{\frac{\Gamma(\alpha)}{\Gamma(2\alpha)}zt}dt.
\end{equation}
Now, recall the Chebyshev integral inequality \cite[p. 40]{DM}: if $f,g:[a,b]\longrightarrow\mathbb{R}$ are synchoronous (both increasing  or decreasing) integrable functions, and $p:[a,b]\longrightarrow\mathbb{R}$  is a positive integrable function, then 
\begin{equation}\label{OO}
\int_a^b p(t)f(t)dt\int_a^b p(t)g(t)dt\leq \int_a^b p(t)dt\int_a^b p(t)f(t)g(t)dt.
\end{equation}
Note that if $f$ and $g$ are asynchronous (one is decreasing and the other is increasing),
then (\ref{OO}) is reversed. For this consider the functions $p,f,g:[0,1]\longrightarrow\mathbb{R}$ defined by:
$$p(t)=1,\;f(t)=\frac{c_{\alpha,\beta}}{\Gamma(\beta)}(1-t^{1/\alpha})^{\beta-\alpha-1}\;\textrm{and}\;\;g(t)=e^{\frac{\Gamma(\alpha)}{\Gamma(2\alpha)}zt}.$$
Since the function $f$ is decreasing and $g$ increasing if  $\beta-\alpha\geq1$. On the other hand, we have
$$\int_0^1f(t)dt=\frac{\alpha c_{\alpha,\beta}}{\Gamma(\beta)}B(\beta-\alpha,\alpha)=\frac{\Gamma(\alpha)}{\Gamma^2(\beta)}, \;\textrm{and}\;\;\int_0^1g(t)dt=\frac{\Gamma(2\alpha)(e^{\frac{\Gamma(2\alpha)z}{\Gamma(\alpha)}}-1)}{\Gamma(\alpha)z}.$$
So,  using the Chebyshev inequality (\ref{OO}) we get inequality (\ref{OOO}).\\
\textbf{b.}. Another use of the Chebyshev integral inequality (\ref{OO}), that is 
$p,f,g:[0,1]\longrightarrow\mathbb{R}$ defined by:
$$p(t)=W_{\alpha,\alpha}(zt),\;f(t)=c_{\alpha,\beta+1}(1-t^{1/\alpha})^{\beta-\alpha}\;\textrm{and}\;g(t)=c_{\alpha,\beta-1}(1-t^{1/\alpha})^{\beta-\alpha-2}.$$
Observe that the functions $f$ and $g$ are decreasing on $(0,\infty)$ for all $\beta-\alpha\geq2.$ Furthermore, by using the Chebyshev inequality (\ref{OO}) and the integral representation (\ref{1}) we have
\begin{equation}
\begin{split}
W_{\alpha,\beta+1}(z)W_{\alpha,\beta-1}(z)&\leq \left(\int_0^1 W_{\alpha,\alpha}(zt)dt\right).\left(c_{\alpha,\beta+1}c_{\alpha,\beta-1}\int_0^1 (1-t^{1/\alpha})^{2\beta-2\alpha-2}W_{\alpha,\alpha}(zt)dt\right)\\
&\leq  \left(\int_0^1 W_{\alpha,\alpha}(zt)dt\right).\left(c_{\alpha,\beta+1}c_{\alpha,\beta-1}\int_0^1 (1-t^{1/\alpha})^{\beta-\alpha-1}W_{\alpha,\alpha}(zt)dt\right)\\
&=\frac{\alpha c_{\alpha,\beta+1}c_{\alpha,\beta-1}}{c_{\alpha,\beta}}W_{\alpha,\alpha+1}(z)W_{\alpha,\beta}(z),
\end{split}
\end{equation}
and consequently (\ref{gg}) as well. Finally, setting in (\ref{gg}) the value $\beta=\alpha+2$ we deduce that the inequality (\ref{ggg}) is hold true. 
\end{proof}

In order to establish a bilateral functional inequalities for $W_{\alpha,\beta}(z)$, we need the Fox--Wright function ${}_p\Psi_q(z)$ defined by
\begin{equation}\label{LLLL}
{}_p\Psi_q\Big[^{(a_1,\alpha_1),...,(a_p,\alpha_p)}_{(b_1,\beta_1),...,(b_q,\beta_q)}\Big|z\Big]=\sum_{k=0}^\infty\frac{\prod_{i=1}^p\Gamma(a_j+\alpha_j k)}{\prod_{j=1}^q\Gamma(b_j+\beta_j k)}\frac{z^k}{k!},
\end{equation}
where $z,a_i,b_j\in\mathbb{C},\;\alpha_i,\beta_j\in\mathbb{R}$ for $i\in\{1,...,p\}$ and $j\in\{1,...,q\}.$ The series (\ref{LLLL}) converges absolutely and uniformly for all bounded $|z|,\;z\in\mathbb{C}$ when
$$1+\sum_{j=1}^q\beta_j-\sum_{i=1}^p\alpha_i>0.$$

We note that the inequality (\ref{777}) in the next Theorem complements and improve the inequality (\ref{23}).

\begin{theorem} \label{T3}Let $\beta>\alpha>0.$ The following inequalities hold true:
\begin{equation}\label{666}
\left(\frac{\Gamma(\alpha)}{\Gamma(\beta)}\right).e^{\frac{\Gamma(2\alpha)\Gamma(\beta)}{\Gamma(\alpha)\Gamma(\alpha+\beta)}|z|}\leq{}_1\Psi_1\Big[^{\;(\alpha,\alpha)\;}_{\;(\beta,\alpha)\;}\;\Big|z\Big]\leq \left(\frac{\Gamma(\alpha)}{\Gamma(\beta)}\right)-\left(\frac{\Gamma(2\alpha)(1-e^{|z|})}{\Gamma(\beta+\alpha)}\right),\;z\in\mathbb{R},
\end{equation} 
\begin{equation}\label{6666}
 W_{\alpha,\beta}(z)\leq \left(\frac{1}{\Gamma(\beta)}\right)-\left(\frac{\Gamma(2\alpha)(1-e^{\frac{\Gamma(\alpha)z}{\Gamma(2\alpha)}})}{\Gamma(\alpha)\Gamma(\beta+\alpha)}\right),\;z>0,
\end{equation}
\begin{equation}\label{777}
 \check{W}_{\alpha,\beta}(z)\geq\frac{e^{\frac{\Gamma(\beta)}{\Gamma(\alpha+\beta)}z}}{\Gamma(\beta)},\;0<z<1.
\end{equation}
\end{theorem}
\begin{proof}
We recall that Pog\'any and Srivastava \cite[Theorem 4]{PS} and \cite[eq. (22)]{PS} which say that for all $_p\Psi_q$ satisfying
\begin{equation}\label{555}
\psi_1>\psi_2\;\;\textrm{and}\;\;\psi_1^2<\psi_0\psi_2,
\end{equation}
the  two--sided inequality
\begin{equation}\label{ùù}
\psi_0 e^{\psi_1\psi_0^{-1}|x|}\leq{}_p\Psi_q\left[_{(b_1,\beta_1),...,(b_q,\beta_q)}^{(a_1,\alpha_1),...,(a_p,\alpha_p)}\Big|x\right]\leq \psi_0-(1-e^{|x|})\psi_1,
\end{equation}
holds true for all $x\in\mathbb{R}.$ Here
$$\psi_m=\frac{\prod_{j=1}^p\Gamma(a_j+\alpha_j m)}{\prod_{j=1}^q\Gamma(b_j+\beta_j m)},\;j\in\left\{1,2,3\right\}.$$
In our case, we have $$\psi_0=\frac{\Gamma(\alpha)}{\Gamma(\beta)},\;\;\psi_1=\frac{\Gamma(2\alpha)}{\Gamma(\beta+\alpha)}\;\textrm{and}\;\psi_2=\frac{\Gamma(3\alpha)}{\Gamma(\beta+2\alpha)}.$$
On the other hand, Due to log--convexity property of the Gamma function $\Gamma(z)$, the ratios $z\mapsto\Gamma(z+a)/\Gamma(z)$ is increasing on $(0,\infty)$ when $a>0$. Thus implies that the following inequality:
\begin{equation}
\frac{\Gamma(z+a)}{\Gamma(z)}\leq \frac{\Gamma(z+a+b)}{\Gamma(z+b)},
\end{equation}
holds for all $a,b,z>0.$ Letting $z=2\alpha,\;a=\alpha$ and $b=\beta-\alpha>0$ in (\ref{55}) we get $\psi_1>\psi_2.$ This proves the left--hand side of inequality (\ref{555}). Now, we consider the function $f:(0,\infty)\longrightarrow\mathbb{R}$ defined by:
$$f_\alpha(z)=\frac{\Gamma(z)\Gamma(z+2\alpha)}{\Gamma^2(\alpha+z)}.$$
Thus,
\begin{equation}\label{yy}
\frac{f^\prime_\alpha(z)}{f_\alpha(z)}=\psi(z)+\psi(z+2\alpha)-2\psi(z+\alpha),
\end{equation}
where $\psi(z)=\Gamma^\prime(z)/\Gamma(z)$ is the Euler digamma function. By using the Legendre's formula
$$\psi(z)=-\gamma+\int_0^1\frac{t^{x-1}-1}{t-1}dt,$$
where $\gamma$ is the Euler--Mascheroni constant, we have 
\begin{equation}\label{yyy}
\frac{f^\prime_\alpha(z)}{f_\alpha(z)}=\int_0^1\frac{t^{x-1}}{t-1}g_\alpha(t)dt,
\end{equation}
where $g_\alpha(t)=1+t^{2\alpha}-2t^\alpha,\;t\in[0,1].$ Thus $g^\prime_\alpha(t)=2\alpha t^{\alpha-1}(t^\alpha-1)\leq 0,$ for all $t\in[0,1],$ consequently the function $t\mapsto g_\alpha(t)$ is decreasing on $[0,1]$ and satisfies $g_\alpha(0)=1$ and $g_\alpha(1)=0.$ So, the function $z\mapsto f_\alpha(z)$ is decreasing on $(0,\infty).$ In particular $f_\alpha(\beta)\leq f_\alpha(\alpha)$, which implies the right hand side of (\ref{555}). Then,
\begin{equation}
\left(\frac{\Gamma(\alpha)}{\Gamma(\beta)}\right).e^{\frac{\Gamma(2\alpha)\Gamma(\beta)}{\Gamma(\alpha)\Gamma(\alpha+\beta)}z}\leq{}_1\Psi_1\Big[^{\;(\alpha,\alpha)\;}_{\;(\beta,\alpha)\;}\;\Big|z\Big]\leq \left(\frac{\Gamma(\alpha)}{\Gamma(\beta)}\right)-\left(\frac{\Gamma(2\alpha)(1-e^{|z|}}{\Gamma(\beta+\alpha)}\right)
\end{equation} 
for all $z\in\mathbb{R}.$ Now, we prove the inequality (\ref{6666}) From the integral representation (\ref{1}) and (\ref{L}), we have
\begin{equation}\label{---}
\begin{split}
 W_{\alpha,\beta}(z)&\leq \frac{c_{\alpha,\beta}}{\Gamma(\alpha)}\int_0^1 (1-t^{1/\alpha})^{\beta-\alpha-1} e^{\frac{\Gamma(\alpha)}{\Gamma(2\alpha)}zt}dt\\
&=\frac{c_{\alpha,\beta}}{\Gamma(\alpha)}\sum_{n=0}^\infty\frac{(\Gamma(\alpha)/\Gamma(2\alpha)z)^n}{n!}\int_0^1 (1-t^{1/\alpha})^{\beta-\alpha-1} t^n dt\\
&=\frac{\alpha c_{\alpha,\beta}}{\Gamma(\alpha)}\sum_{n=0}^\infty\frac{(\Gamma(\alpha)/\Gamma(2\alpha)z)^n}{n!}\int_0^1 (1-t)^{\beta-\alpha-1} t^{\alpha n+\alpha-1} dt\\
&=\frac{\alpha c_{\alpha,\beta}}{\Gamma(\alpha)}\sum_{n=0}^\infty\frac{B(\beta-\alpha,\alpha n+\alpha)(\Gamma(\alpha)/\Gamma(2\alpha)z)^n}{n!}\\
&=\frac{1}{\Gamma(\alpha)}\sum_{n=0}^\infty\frac{\Gamma(\alpha n+\alpha)(\Gamma(\alpha)/\Gamma(2\alpha)z)^n}{n!\Gamma(\alpha n+\beta)}\\
&=\frac{1}{\Gamma(\alpha)}{}_1\Psi_1\Big[^{\;(\alpha,\alpha)\;}_{\;(\beta,\alpha)\;}\;\Big|\frac{\Gamma(\alpha)}{\Gamma(2\alpha)}z\Big].
\end{split}
\end{equation}
So, by the right hand side of inequality (\ref{666}) and (\ref{---}) we deduce that the inequality (\ref{6666}) holds true for all $z>0.$  Similar arguments would lead us to proved the inequality (\ref{777}). By means of the integral representation (\ref{1}) and the inequality (\ref{23}) we have  
\begin{equation}\label{7777}
\begin{split}
 W_{\alpha,\beta}(z)&\geq \frac{c_{\alpha,\beta}}{\Gamma(\alpha)}\int_0^1 (1-t^{1/\alpha})^{\beta-\alpha-1} e^{-\frac{\Gamma(\alpha)}{\Gamma(2\alpha)}zt}dt\\
&=\frac{c_{\alpha,\beta}}{\Gamma(\alpha)}\sum_{n=0}^\infty\frac{(-\Gamma(\alpha)/\Gamma(2\alpha)z)^n}{n!}\int_0^1 (1-t^{1/\alpha})^{\beta-\alpha-1} t^n dt\\
&=\frac{1}{\Gamma(\alpha)}{}_1\Psi_1\Big[^{\;(\alpha,\alpha)\;}_{\;(\beta,\alpha)\;}\;\Big|-\frac{\Gamma(\alpha)}{\Gamma(2\alpha)}z\Big].
\end{split}
\end{equation}
 Combining the left hand side of inequality (\ref{666}) and (\ref{7777}) we obtain the inequality (\ref{777}). This evidently completes the proof of Theorem \ref{T3}.
\end{proof}
\section{The generalized Wright functions: Monotonicity patterns and functional inequalities}
 In \cite{AS}, the authors introduced the definition of the generalized Wright function $W_{\alpha,\beta}^{\gamma,\sigma}(z)$:
\begin{equation}\label{def}
W_{\alpha,\beta}^{\gamma,\sigma}(z)=\sum_{n=0}^\infty\frac{(\gamma)_n}{(\sigma)_n\Gamma(\alpha n+\beta)}\frac{z^n}{n!},\;\alpha\in\mathbb{R},\beta,\gamma,\sigma,z\in\mathbb{C},
\end{equation}
where 
$$(\tau)_n=\frac{\Gamma(\tau+n)}{\Gamma(\tau)}=\tau(\tau+1)...(\tau+n-1),$$
is a Pochhammer symbol. The function $W_{\alpha,\beta}^{\gamma,\sigma}(z)$ is an entire function of order $1/(1+\alpha)$ and has the following integral representation \cite[Theorem 2, eq. (34)]{AS} 
\begin{equation}\label{&&&}
W_{\alpha,\beta}^{\gamma,\sigma}(z)=\frac{\Gamma(\sigma)}{\Gamma(\gamma)\Gamma(\sigma-\gamma)}\int_0^1 t^{\gamma-1}(1-t)^{\sigma-\gamma-1}W_{\alpha,\beta}(zt)dt,
\end{equation}
where $\alpha>-1, \beta, \gamma,\sigma,z\in\mathbb{C}$ and $\Re(\sigma)>\Re(\gamma)>0.$
\begin{theorem}\label{T4}The following assertions are true:\\
\textbf{a.} The function $z\mapsto \hat{W}_{\alpha,\beta}^{\gamma,\sigma}(z)$ is completely monotonic and log--convex on $(0,1),$ for all $\alpha,\gamma,\sigma>0$ such that $\beta>\alpha>x^\star$ and $\sigma>\gamma.$ Moreover, the following inequalities holds true:
\begin{equation}\label{ZZZ0}
\check{W}_{\alpha,\beta}^{\gamma,\sigma}(x+y)\geq\frac{\check{W}_{\alpha,\beta}^{\gamma,\sigma}(x)\check{W}_{\alpha,\beta}^{\gamma,\sigma}(y)}{\Gamma(\beta)},\;\;0<x+y<1,
\end{equation}
\begin{equation}\label{ZZZ1}
\frac{\gamma+1}{\sigma+1}\hat{W}_{\alpha,\beta+2\alpha}^{\gamma+2,\sigma+2}(z)\hat{W}_{\alpha,\beta}^{\gamma,\sigma}(z)-\frac{\gamma}{\sigma}.(\hat{W}_{\alpha,\beta+\alpha}^{\gamma+1,\sigma+1}(z))^2\geq 0,\;0<z<1,
\end{equation}
\begin{equation}\label{ZZZ2}
\check{W}_{\alpha,\beta}^{\gamma,\sigma}(z)\geq \frac{e^{-\frac{\gamma\Gamma(\beta)}{\sigma\Gamma(\beta+\alpha)}z}}{\Gamma(\beta)},\;z\in(0,1).
\end{equation}
\textbf{b.} The function $\sigma\mapsto W_{\alpha,\beta}^{\gamma,\sigma}(z)$ is log--convex on $(0,\infty).$ Moreover, the following Tur\'an type inequality 
\begin{equation}\label{ZZZ3}
W_{\alpha,\beta}^{\gamma,\sigma}(z)W_{\alpha,\beta}^{\gamma,\sigma+2}(z)-\Big(W_{\alpha,\beta}^{\gamma,\sigma+1}(z)\Big)^2\geq 0.
\end{equation}
\end{theorem}
\begin{proof}\textbf{a.} From Theorem \ref{T1} and integral representation of the generalized Wright function $W_{\alpha,\beta}^{\gamma,\sigma}(z)$, we deduce that the function $z\mapsto \check{W}_{\alpha,\beta}^{\gamma,\sigma}(z)$ is completely monotonic on $(0,1)$ and consequently is log--convex. Again using the Kimberling's result, we obtain the inequality (\ref{ZZZ0}). Now, we prove the inequality (\ref{ZZZ1}). Since the function $z\mapsto \hat{W}_{\alpha,\beta}^{\gamma,\sigma}(z)$ is log--convex on $(0,1)$ we have  $z\mapsto (\hat{W}_{\alpha,\beta}^{\gamma,\sigma}(z))^\prime/\hat{W}_{\alpha,\beta}^{\gamma,\sigma}(z)$  is increasing on $(0,1)$. So, by using the differentiation formula \cite[Theorem  19]{AS}
\begin{equation}
\frac{d}{dz}W_{\alpha,\beta}^{\gamma,\sigma}(z)=\frac{\gamma}{\sigma}W_{\alpha,\beta+\alpha}^{\gamma+1,\sigma+1}(z),
\end{equation}
we get
\begin{equation}
\left(\frac{(\hat{W}_{\alpha,\beta}^{\gamma,\sigma}(z))^\prime}{\hat{W}_{\alpha,\beta}^{\gamma,\sigma}(z)}\right)^\prime=\frac{\frac{\gamma(\gamma+1)}{\sigma(\sigma+1)}\hat{W}_{\alpha,\beta+2\alpha}^{\gamma+2,\sigma+2}(z)\hat{W}_{\alpha,\beta}^{\gamma,\sigma}(z)-\frac{\gamma^2}{\sigma^2}.(\hat{W}_{\alpha,\beta+\alpha}^{\gamma+1,\sigma+1}(z))^2}{(\hat{W}_{\alpha,\beta}^{\gamma,\sigma}(z))^2}\geq0,
\end{equation}
which can be derived easily the inequality (\ref{ZZZ1}). Now, we prove the inequality (\ref{ZZZ2}). Let $F_1(x)=\log\Big[\Gamma(\beta)\hat{W}_{\alpha,\beta}^{\gamma,\sigma}(z)\Big]$ and $G_1(x)=x.$ Again by using the  monotone form of l'Hospital's rule, we deduce that the function $F_1(x)/G_1(x)=(F_1(x)-F_1(0))/(G_1(x)-G_1(0))$ is increasing on $(0,1),$ and consequently 
$$\lim_{x\longrightarrow 0}\frac{F_1(x)}{G_1(x)}=-\frac{\gamma\Gamma(\beta)}{\sigma\Gamma(\beta+\alpha)},$$
which completes the proof of inequality (\ref{ZZZ2}).\\
\textbf{b.} For convenience, let us write $A_n(\sigma)=\frac{(\gamma)_n}{(\sigma)_n n!\Gamma(\alpha n+\beta)}.$ Since the function $\psi^\prime$ is completely monotonic on $(0,\infty)$ we get
$$\partial^2 [\log A_n(\sigma)]/\partial\sigma^2=\psi^\prime(\sigma)-\psi^\prime(\sigma+n)\geq0,$$
for all $n\geq0.$ So, using  the fact that sums of log--convex functions are log--convex too, we deduce that the function $\sigma\mapsto W_{\alpha,\beta}^{\gamma,\sigma}(z)$ is log--convex on $(0,\infty),$ for $z>0.$ Now, focus the Tur\'an type inequality (\ref{ZZZ3}). Since $\sigma\mapsto W_{\alpha,\beta}^{\gamma,\sigma}(z)$ is log--convex on $(0,\infty)$ for $z>0,$ it follows that for $\sigma_1,\sigma_2>0,\;t\in[0,1],$ we have
$$W_{\alpha,\beta}^{\gamma,t\sigma_1+(1-t)\sigma_2}(z)\leq \Big[W_{\alpha,\beta}^{\gamma,\sigma_1}(z)\Big]^t \Big[W_{\alpha,\beta}^{\gamma,\sigma}(z)\Big]^{1-t}.$$
Choosing $\sigma_1=\sigma,\;\sigma_2=\sigma+2$ and $t=1/2$,  the above inequality reduces to the Tur\'an type inequality (\ref{ZZZ3}). The proof of Theorem \ref{T4} is thus completed.
\end{proof}

\begin{theorem}\label{tèè}Let $\beta,\alpha,\sigma>0$ and $\gamma>0$. Then, the following Tur\'an type inequality  
\begin{equation}
 W_{\alpha,\beta}^{\gamma,\sigma}(z) W_{\alpha,\beta}^{\gamma+2,\sigma}(z)-\frac{\gamma}{\gamma+1} \Big(W_{\alpha,\beta}^{\gamma+1,\sigma}(z)\Big)^2\geq0,
\end{equation}
hold true for all $z>0.$
\end{theorem}
\begin{proof}For convenience, let us write $\textsc{K}(\gamma)=\frac{\Gamma(\gamma)}{\Gamma(\sigma)}W_{\alpha,\beta}^{\gamma,\sigma}(z).$ By applying the Cauchy product, we find that
\begin{equation}
\textsc{K}^2(\gamma+1)-\textsc{K}(\gamma)\textsc{K}(\gamma+2)=\sum_{k=0}^\infty\sum_{j=0}^k \delta_{j,k}T_{j,k}z^k,
\end{equation}
where $T_{j,k}=((2j-k)-1)\Gamma(\gamma+j)\Gamma(\gamma+(k-j)+1)$ and $\delta_{j,k}=1/(j!(k-j)!\Gamma(\sigma+j)\Gamma(\sigma+k-j)\Gamma(\alpha j+\beta)\Gamma(\alpha (k-j)+\beta)).$ If $k$ is even, then
\begin{equation*}
\begin{split}
 \sum_{j=0}^k \delta_{j,k}T_{j,k}&=\sum_{j=0}^{k/2-1} \delta_{j,k}T_{j,k}+\sum_{j=0}^{k/2+1} \delta_{j,k}T_{j,k}+\delta_{\frac{k}{2},k}T_{\frac{k}{2},k}\\
&=\sum_{j=0}^{k/2-1} \delta_{j,k}T_{j,k}+\sum_{j=0}^{k/2-1} \delta_{j,k}T_{k-j,k}+\delta_{\frac{k}{2},k}T_{\frac{k}{2},k}\\
&=\sum_{j=0}^{[(k-1)/2]} \delta_{j,k}\Big(T_{j,k}+T_{k-j,k}\Big)-\delta_{\frac{k}{2},k}\Gamma(\gamma+k/2)\Gamma(\gamma+k/2+1),
\end{split}
\end{equation*}
where, as usual, $[k]$ denotes the greatest integer part of $k\in\mathbb{R}.$  Similarly, if $k$ is odd, then
 $$\sum_{j=0}^k \delta_{j,k}T_{j,k}=\sum_{j=0}^{[(k-1)/2]} \delta_{j,k}\Big(T_{j,k}+T_{k-j,k}\Big)-\delta_{\frac{k}{2},k}\Gamma(\gamma+k/2)\Gamma(\gamma+k/2+1).$$
Therefore,
$$\textsc{K}^2(\gamma+1)-\textsc{K}(\gamma)\textsc{K}(\gamma+2)=\sum_{k=0}^\infty\sum_{j=0}^{[(k-1)/2]} \delta_{j,k}\Big(T_{j,k}+T_{k-j,k}\Big)-\delta_{\frac{k}{2},k}\Gamma(\gamma+k/2)\Gamma(\gamma+k/2+1).$$
Simplifying, we find  that 
\begin{equation*}
\begin{split}
T_{j,k}+T_{k-j,k}&=(k-2j)((2j-k)-1)\Gamma(\gamma+j)\Gamma(\gamma+(k-j))\leq0,
\end{split}
\end{equation*}
for $k<k-j$ (i.e $[(k-1)/2]\geq j$), which evidently completes the proof of Theorem \ref{tèè}.
\end{proof}
\begin{theorem}\label{T5}Let $\beta>\alpha>0$ and $\sigma>\gamma>0.$ Then the following inequalities
\begin{equation}\label{8}
\frac{\Gamma(\gamma)}{\Gamma(\sigma)} e^{\frac{\gamma}{\sigma}|z|}\leq {}_1\Psi_1\left[_{(\sigma, 1)}^{(\gamma, 1)}\Big| z\right]\leq \left(\frac{\Gamma(\gamma)}{\Gamma(\sigma)}\right).\left(1-\frac{\gamma}{\sigma}(1-e^{|z|})\right),\;z\in\mathbb{R}
\end{equation}
\begin{equation}\label{88}
W_{\alpha,\beta}^{\gamma,\sigma}(z)\leq \left(\frac{1}{\Gamma(\beta)}\right).\left[1-\frac{\gamma}{\sigma}\left(1-e^{\frac{\Gamma(\beta)}{\Gamma(\beta+\alpha)}z}\right)\right],z>0,
\end{equation}
\begin{equation}\label{888}
\check{W}_{\alpha,\beta}^{\gamma,\sigma}(z)\geq \left(\frac{1}{\Gamma(\beta)}\right). e^{\frac{\gamma\Gamma(\beta)}{\sigma\Gamma(\beta+\alpha)}z}, 0<z<1.
\end{equation}
\end{theorem}
\begin{proof}In our case, we have $\psi_0=\frac{\Gamma(\gamma)}{\Gamma(\sigma)},\psi_1=\frac{\Gamma(\gamma+1)}{\Gamma(\sigma+1)}$ and $\psi_2=\frac{\Gamma(\gamma+2)}{\Gamma(\sigma+2)}$. Since $\sigma>\gamma$, we get $\psi_1>\psi_2$ and $\psi_1^2<\psi_0\psi_2,$ and consequently the conditions (\ref{555}) holds. Then, by using (\ref{ùù}) we deduce that the inequality (\ref{8}) hold true. Next, we prove the inequality (\ref{88}). Combining the inequality (\ref{L}) and the representation integral of the generalized Wright function (\ref{&&&}), we get
\begin{equation}
\begin{split}
W_{\alpha,\beta}^{\gamma,\sigma}(z)&\leq\frac{\Gamma(\sigma)}{\Gamma(\beta)\Gamma(\gamma)\Gamma(\sigma-\gamma)}\int_0^1 t^{\gamma-1}(1-t)^{\sigma-\gamma-1}e^{\frac{\Gamma(\beta)}{\Gamma(\alpha+\beta)}zt}dt\\
&=\frac{\Gamma(\sigma)}{\Gamma(\beta)\Gamma(\gamma)\Gamma(\sigma-\gamma)}\int_0^1 t^{\gamma+n-1}(1-t)^{\sigma-\gamma-1}\left(\sum_{n=0}^\infty \frac{\Big((\Gamma(\beta)/\Gamma(\beta+\alpha))z\Big)^n}{n!}\right)dt\\
&=\frac{\Gamma(\sigma)}{\Gamma(\beta)\Gamma(\gamma)\Gamma(\sigma-\gamma)}\sum_{n=0}^\infty \frac{\Big((\Gamma(\beta)/\Gamma(\beta+\alpha))z\Big)^n}{n!}\int_0^1 t^{\gamma+n-1}(1-t)^{\sigma-\gamma-1}dt\\
&=\frac{\Gamma(\sigma)}{\Gamma(\beta)\Gamma(\gamma)\Gamma(\sigma-\gamma)}\sum_{n=0}^\infty \frac{B(\gamma+n,\sigma-\gamma)\Big((\Gamma(\beta)/\Gamma(\beta+\alpha))z\Big)^n}{n!}\\
&=\frac{\Gamma(\sigma)}{\Gamma(\beta)\Gamma(\gamma)}\sum_{n=0}^\infty \frac{\Gamma(\gamma+n)\Big((\Gamma(\beta)/\Gamma(\beta+\alpha))z\Big)^n}{\Gamma(\sigma+n)n!}\\
&=\frac{\Gamma(\sigma)}{\Gamma(\beta)\Gamma(\gamma)} {}_1\Psi_1\Big[^{(\gamma,1)}_{(\sigma,1)}\Big|\frac{\Gamma(\beta)}{\Gamma(\beta+\alpha)}z\Big].
\end{split}
\end{equation}
Combining this equation with the right hand side of inequalities  (\ref{8}), we obtain (\ref{88}). It remains to prove (\ref{888}). The integral representation (\ref{&&&}) of the function $W_{\alpha,\beta}^{\gamma,\sigma}(z)$ and inequality (\ref{ZZZ2}) yields that 
\begin{equation}
\begin{split}
\check{W}_{\alpha,\beta}^{\gamma,\sigma}(z)&\geq\frac{\Gamma(\sigma)}{\Gamma(\beta)\Gamma(\gamma)\Gamma(\sigma-\gamma)}\int_0^1 t^{\gamma-1}(1-t)^{\sigma-\gamma-1}e^{-\frac{\Gamma(\beta)}{\Gamma(\alpha+\beta)}zt}dt\\
&=\frac{\Gamma(\sigma)}{\Gamma(\beta)\Gamma(\gamma)\Gamma(\sigma-\gamma)}\int_0^1 t^{\gamma+n-1}(1-t)^{\sigma-\gamma-1}\left(\sum_{n=0}^\infty \frac{\Big(-(\Gamma(\beta)/\Gamma(\beta+\alpha))z\Big)^n}{n!}\right)dt\\
&=\frac{\Gamma(\sigma)}{\Gamma(\beta)\Gamma(\gamma)\Gamma(\sigma-\gamma)}\sum_{n=0}^\infty \frac{\Big(-(\Gamma(\beta)/\Gamma(\beta+\alpha))z\Big)^n}{n!}\int_0^1 t^{\gamma+n-1}(1-t)^{\sigma-\gamma-1}dt\\
&=\frac{\Gamma(\sigma)}{\Gamma(\beta)\Gamma(\gamma)} {}_1\Psi_1\left[^{(\gamma,1)}_{(\sigma,1)}\Bigg|-\frac{\Gamma(\beta)}{\Gamma(\beta+\alpha)}z\right].
\end{split}
\end{equation}
From the above inequality and the left hand side of inequalities (\ref{8}) we deduce (\ref{888}) for all $0<z<1$ and  $\beta>\alpha>0$ and $\sigma>\gamma>0.$  The proof of Theorem \ref{T5} is completes.
\end{proof}
\begin{remark} We point out that  the inequality (\ref{888}) complements and improve the inequality (\ref{ZZZ2}). Since $e^z\geq e^{-z}$ for all $z>0,$ we deduce that the inequality (\ref{888}) is better than (\ref{ZZZ2}).
\end{remark}
\begin{theorem}\label{T6} The following inequalities holds true:\\
\noindent \textbf{a.} For all $z>0,\;0<\gamma\leq 1$ and $\sigma-\gamma\geq1$, we have
\begin{equation}\label{1010}
W_{\alpha,\beta}^{\gamma,\sigma}(z)\leq \frac{\Gamma(\beta-\alpha)W_{\alpha,\beta-\alpha}(z)-1}{\Gamma(\beta-\alpha)z}=W_{\alpha,\beta}^{1,2}(z).
\end{equation}
\noindent \textbf{b.} For all $z>0,\;0<\gamma\leq 1$ and $\sigma-\gamma\geq2$, we have
\begin{equation}\label{11111}
W_{\alpha,\beta}^{\gamma,\sigma+1}(z)W_{\alpha,\beta}^{\gamma, \sigma-1}(z)\leq \frac{\Gamma(\sigma-\gamma)\Gamma(\sigma+1)\Gamma(\sigma-1)}{\Gamma(\sigma)\Gamma(\gamma)\Gamma(\sigma-\gamma+1)\Gamma(\sigma-\gamma-1)}W_{\alpha,\beta}^{1,2}(z)W_{\alpha,\beta}^{\gamma, \sigma}(z).
\end{equation}
\end{theorem}
\begin{proof}\textbf{a.} By again using the Chebyshev integral inequality (\ref{OO}), we consider the functions $p,f,g:[0,1]\longrightarrow\mathbb{R}$ defined by  $$p(t)=1,\;f(t)=(B(\sigma-\gamma,\gamma))^{-1} (1-t)^{\sigma-\gamma-1}t^{\gamma-1}\;\textrm{and}\;\;g(t)=W_{\alpha,\beta}(zt).$$
Observe that the function $f(t)$ is decreasing and $g(t)$ is increasing on $[0,1],$ if $0<\gamma\leq 1$ and $\sigma-\gamma\geq1.$ On the other hand,
$$\int_0^1p(t)f(t)dt=1,\;\;\textrm{and}\;\;\int_0^1p(t) g(t)=\frac{1}{z}\int_0^1 \left(W_{\alpha,\beta-\alpha}(zt)\right)^\prime dt=\frac{1}{z}\left(W_{\alpha,\beta-\alpha}(z)-\frac{1}{\Gamma(\beta-\alpha)}\right)=W_{\alpha,\beta}^{1,2}(z).$$
So, the integral representation (\ref{&&&}) completes the proof of inequality (\ref{1010}). \\
\textbf{b.} For the proof of inequality (\ref{11111}), we consider the functions $p,f,g:[0,1]\longrightarrow\mathbb{R}$ defined by 
$$p(t)=W_{\alpha,\beta}(zt),\;f(t)=(1-t)^{\sigma-\gamma}t^{\gamma-1},\;\;\textrm{and}\;\;g(t)= (1-t)^{\sigma-\gamma-2}t^{\gamma-1}.$$
Thus,
$$\int_0^1p(t)f(t)dt=\frac{\Gamma(\gamma)\Gamma(\sigma-\gamma+1)}{\Gamma(\sigma+1)}W_{\alpha,\beta}^{\gamma, \sigma+1}(z),\;\;\int_0^1p(t) g(t)=\frac{\Gamma(\gamma)\Gamma(\sigma-\gamma-1)}{\Gamma(\sigma-1)}W_{\alpha,\beta}^{\gamma, \sigma-1}(z),$$
and
$$\int_0^1p(t)dt=\int_0^1W_{\alpha,\beta}(zt)dt=W_{\alpha,\beta}^{1,2}(zt)=\frac{1}{z}\left(W_{\alpha,\beta}(z)-\frac{1}{\Gamma(\beta-\alpha)}\right).$$
On the other hand, the functions $f(t)$ and $g(t)$ are decreasing on $[0,1]$ if $0<\gamma\leq 1$ and $\sigma-\gamma\geq2.$ Therefore, the Chebyshev integral inequality (\ref{OO}) yields that 
\begin{equation*}
\begin{split}
\frac{\Gamma^2(\gamma)\Gamma(\sigma-\gamma+1)\Gamma(\sigma-\gamma-1)}{\Gamma(\sigma+1)\Gamma(\sigma-1)}W_{\alpha,\beta}^{\gamma, \sigma+1}(z)W_{\alpha,\beta}^{\gamma, \sigma-1}(z)&=\left(\int_0^1p(t)f(t)dt\right).\left(\int_0^1p(t)g(t)dt\right)\\
&\leq \left(\int_0^1W_{\alpha,\beta}(zt)dt\right).\left(\int_0^1 (1-t)^{2\sigma-2\gamma-2}t^{2\gamma-2}W_{\alpha,\beta}(zt)\right)dt\\
&\leq \left(\int_0^1W_{\alpha,\beta}(zt)dt\right).\left(\int_0^1 (1-t)^{\sigma-\gamma-1}t^{\gamma-1}W_{\alpha,\beta}(zt)\right)dt\\
&=\frac{\Gamma(\gamma)\Gamma(\sigma-\gamma)}{\Gamma(\sigma)}W_{\alpha,\beta}^{1,2}(z)W_{\alpha,\beta}^{\gamma, \sigma}(z).
\end{split}
\end{equation*}
The proof of Theorem \ref{T6} is completes.
\end{proof}
\begin{remark}We note that the  results obtained  in section 3 is not a generalization of the results obtained in section 2. except Theorem \ref{T4}, assertion \textbf{a.} and equations (\ref{ZZZ0}), (\ref{ZZZ1}) and (\ref{ZZZ2}).  Indeed, the results in section 3 follows by using the new integral representation (\ref{1}) and the results of section 3 follows by using the integral representation (\ref{&&&}) which is  different from the integral representation (\ref{1}). Then, in the same way we obtain that the function $W_{\alpha,\beta}^{\sigma,\gamma}(z)$ admits this integral representation
\begin{equation}
W_{\alpha,\beta}^{\gamma, \sigma}(z)=c_{\alpha,\beta}\int_0^1(1-t^{1/\alpha})^{\beta-\alpha-1}W_{\alpha,\alpha}^{\gamma,\sigma}(zt)dt,
\end{equation} 
which is a generalization of (\ref{1}), and consequently we can obtain the generalization of some results in section 2.
\end{remark}
\section{Applications: Monotonicity patterns and functional inequalities for the four--parametric Mittag--Leffler functions}
The Mittag--Leffler functions with $2n$  parameters are defined for $B_j\in\mathbb{R}\;(B_1^2+...+B_n^2\neq0)$ and $\beta_j\in\mathbb{C}\;(j=1,...,n\in\mathbb{N})$ by the series
\begin{equation}\label{197}
E_{(B,\beta)_n}(z)=\sum_{k=0}^\infty\frac{z^k}{\prod_{j=1}^n\Gamma(\beta_j+kB_j)},\;z\in\mathbb{C}.
\end{equation}
When $n = 1$, the deﬁnition in (\ref{197})  coincides with the definition of the two--parametric Mittag--Leffler function
\begin{equation}\label{198}
E_{(B,\beta)_1}(z)=E_{B, \beta}(z)=\sum_{k=0}^\infty\frac{z^k}{\Gamma(\beta+kB)},\;z\in\mathbb{C},
\end{equation}
and and similarly for $n = 2$, where $E_{(B, \beta)_2}(z)$  coincides with the four--parametric Mittag--Leffler function
\begin{equation}\label{199}
E_{(B,\beta)_2}(z)=E_{B_1, \beta_1; B_2\beta_2}(z)=\sum_{k=0}^\infty\frac{z^k}{\Gamma(\beta_1+kB_1)\Gamma(\beta_2+kB_2)},\;z\in\mathbb{C},
\end{equation}
is closer by its properties to the Wright function $W_{B,\beta}(z)$ defined by
\begin{equation}
W_{B,\beta}(z)=\sum_{k=0}^\infty\frac{z^k}{k!\Gamma(\beta_1+kB_1))},\;z\in\mathbb{C}.
\end{equation}
The generalized $2n-$parametric Mittag-Leffler function $E_{(\beta,B)_n}(z)$ can be represented in terms of the  Fox--Wright hypergeometric function ${}_p\Psi_q(z)$ by
\begin{equation}\label{199}
E_{(B,\beta)_n}(z)={}_1\Psi_n\Big[_{(\beta_1,B_1),...,(\beta_q,B_q)}^{\;\;\;\;\;\;\;\;(1,1)\;\;\;\;}\Big|z \Big],\;z\in\mathbb{C}.
\end{equation}
Letting $\gamma=1$ in definition (\ref{def}) of the generalized Wright function, we obtain that 
\begin{equation}
W_{\alpha,\beta}^{1,\sigma}(z)=\Gamma(\sigma)E_{\alpha,\beta;1,\sigma}(z),
\end{equation}
and consequently we obtain the following assertions for the four--parametric Mittag--Leffler function $E_{\alpha,\beta;1,\sigma}(z)$:
\begin{theorem}\textbf{a.} The function $z\mapsto E_{\alpha,\beta;1,\sigma}(-z)=\check{E}_{\alpha,\beta;1,\sigma}(z)$ is completely monotonic and log--convex on $(0,1)$ for all $\beta>\alpha>x^\star$ and $\sigma>1.$
Furthermore, the following inequalities hold true:
\begin{equation}
\check{E}_{\alpha,\beta;1,\sigma}(x+y)\geq \left(\frac{\Gamma(\sigma)}{\Gamma(\beta)}\right).\check{E}_{\alpha,\beta;1,\sigma}(x)\check{E}_{\alpha,\beta;1,\sigma}(y),\;0<x+y<1.
\end{equation}
\begin{equation}
\frac{2}{\sigma+1}\check{E}_{\alpha,\beta+2\alpha;3,\sigma+2}(z)\check{E}_{\alpha,\beta;1,\sigma}(z)-\frac{1}{\sigma}\Big(\check{E}_{\alpha,\beta+\alpha;2,\sigma+1}(z)\Big)^2\geq0,\;0<z<1.
\end{equation}
\begin{equation}
\check{E}_{\alpha,\beta;1,\sigma}(z)\geq \frac{e^{\frac{\Gamma(\beta)}{\sigma\Gamma(\beta+\alpha)}z}}{\Gamma(\sigma)}, 0<z<1.
\end{equation}
\textbf{b.} The function $\sigma\mapsto\Gamma(\sigma)E_{\alpha,\beta;1,\sigma}(z)$ is log-convex on $(0,\infty)$ for all $z,\alpha,\beta>0.$ Moreover, the following Tur\'an type inequality
$$E_{\alpha,\beta;1,\sigma+2}(z)E_{\alpha,\beta;1,\sigma}(z)-\frac{\sigma}{\sigma+1}\Big(E_{\alpha,\beta;1,\sigma+1}(z)\Big)^2\geq0,$$
hold true for all $z,\alpha,\beta>0.$\\
\noindent\textbf{c.} Let $\beta>\alpha>0$ and $\sigma>1.$ Then, the following inequality
$$E_{\alpha,\beta;1,\sigma}(z)\leq \left(\frac{\Gamma(\sigma)}{\Gamma(\beta)}\right).\left[1-\frac{1}{\sigma}\left(1-e^{\frac{\Gamma(\beta)}{\Gamma(\beta+\alpha)}z}\right)\right],$$
hold true for all $z>0.$\\
\noindent\textbf{d.} Let $\beta>\alpha>0$ and $\sigma>1.$ Then
$$E_{\alpha,\beta;1,\sigma+1}(z)E_{\alpha,\beta;1,\sigma-1}(z)\leq\frac{\Gamma(\sigma-1)}{\Gamma(\sigma)\Gamma(\sigma-2)}E_{\alpha,\beta;1,2}(z)E_{\alpha,\beta;1,\sigma}(z),$$
hold for all $z>0$ and $\sigma\geq3.$
\end{theorem}

\end{document}